\documentclass[11pt,reqno]{amsart}
\usepackage{color}

\usepackage{amsfonts,amscd}
\usepackage{amssymb}
\usepackage{url}
\usepackage[english]{babel}

\theoremstyle{plain}
\newtheorem{theorem}                 {Theorem}      [section]
\newtheorem{conjecture}   [theorem]  {Conjecture}

\newtheorem{lemma}        [theorem]  {Lemma}
\newtheorem{proposition}  [theorem]  {Proposition}

\theoremstyle{definition}

\newtheorem{definition}   [theorem]  {Definition}

\newtheorem{example}      [theorem]  {Example}

\newtheorem{remark}       [theorem]  {Remark}

\numberwithin{equation}{section}

\def \cn{{\mathbb C}}

\def \rn{{\mathbb R}}

\def \zn{{\mathbb Z}}

\def \B{\mathcal B}

\def\nab#1#2{\hbox{$\nabla$\kern -.3em\lower 1.0 ex
		\hbox{$#1$}\kern -.1 em {$#2$}}}

\def \SL2{\widetilde{\text{\bf SL}}_{2}(\rn)}

\DeclareMathOperator{\Div}{div}

\numberwithin{equation}{section}
\allowdisplaybreaks

\begin{document}

\title[$(p,q)$-harmonic morphisms]{Complex-valued $(p,q)$-harmonic morphisms\\ from Riemannian manifolds}

\dedicatory{* The corresponding author}

\author{Elsa Ghandour}
\address{Mathematics, Faculty of Science\\
	University of Lund\\
	Box 118, Lund 221 00\\
	Sweden}
\email{Elsa.Ghandour@math.lu.se}

\author{Sigmundur Gudmundsson*}
\address{Mathematics, Faculty of Science\\
	University of Lund\\
	Box 118, Lund 221 00\\
	Sweden}
\email{Sigmundur.Gudmundsson@math.lu.se}

\begin{abstract}
We introduce the natural notion of $(p,q)$-harmonic morphisms between Riemannian manifolds.  This unifies several theories that have been studied during the last decades.  We then study the special case when the maps involved are complex-valued.  For these we find a  characterisation and provide new non-trivial examples in important cases.
\end{abstract}



\maketitle

\section{Introduction}\label{section-introduction}

The history of {\it harmonic morphisms} can be traced back to the pioneering work \cite{Jac} of Jacobi from 1848.  He studies complex-valued functions pulling back harmonic functions in the complex plane $\cn$ to harmonic functions in the $3$-dimensional Euclidean space $\rn^3$.  The notion is then generalised to the Riemannian setting in the late 1970s, independently by Fuglede and Ishihara, see \cite{Fug} and \cite{Ish}.  This has led to a lively development that can be followed both in \cite{Bai-Woo-book} and at the regularly updated on-line bibliography \cite{Gud-bib}.

Loubeau and Ou study {\it biharmonic morphisms} between Riemannian manifolds, see \cite{Lou-Ou-1} and \cite{Lou-Ou-2}.  These are maps pulling back biharmonic functions to biharmonic functions.  In his work \cite{Mae}, Maeta introduces the notion of {\it triharmonic morphisms}.  These are mappings pulling back triharmonic functions to triharmonic functions.

Recently, Ghandour and Ou introduce the notion of {\it generalised harmonic morphisms} between Riemannian manifolds, see \cite{Gha-thesis} and \cite{Gha-Ou-1}.  These are maps pulling back harmonic functions to biharmonic functions.  They also find a characterisation for these non-linear objects.
\medskip

In this work we unify the above notions by defining the {\it $(p,q)$-harmonic morphisms}.  These are maps between Riemannian manifolds pulling back $q$-harmonic functions to $p$-harmonic  functions.  Just as the harmonic morphisms and their above mentioned variants, they are solutions to an over-determined system of {\it non-linear} partial differential equations.  This means that they have no general existence theory.  For this reason it is interesting to develop methods for constructing solutions in particular cases.

In this paper we focus our attention on complex-valued $(p,q)$-harmonic morphisms from Riemannian manifolds. The aim is to extend the known characterisation to this case and to manufacture new non-trivial examples to this non-linear problem.  The explicit examples presented here involve rather demanding computations.  They were all tested, by the computer algebra systems Maple and Mathematica, independently.

\section{Preliminaries}\label{section-preliminaries}

Let $(M,g)$ be an $m$-dimensional  Riemannian manifold and $T^{\cn}M$ be the complexification of the tangent bundle $TM$ of $M$. We extend the metric $g$ to a complex-bilinear form on $T^{\cn}M$.  Then the gradient $\nabla z$ of a complex-valued function $z:(M,g)\to\cn$ is a section of $T^{\cn}M$.  In this situation, the well-known complex linear {\it Laplace-Beltrami operator} (alt. {\it tension field}) $\tau$ on $(M,g)$ acts locally on $z$ as follows
$$
\tau(z)=\Div (\nabla z)=\sum_{i,j=1}^m\frac{1}{\sqrt{|g|}} \frac{\partial}{\partial x_j}
\left(g^{ij}\, \sqrt{|g|}\, \frac{\partial z}{\partial x_i}\right).
$$
For two complex-valued functions $z,w:(M,g)\to\cn$ we have the following well-known relation
\begin{equation}\label{equation-basic}
\tau(z\cdot w)=\tau(z)\cdot w +2\cdot\kappa(z,w)+z\cdot\tau(w),
\end{equation}
where the complex bilinear {\it conformality operator} $\kappa$ is given by $\kappa(z,w)=g(\nabla z,\nabla w)$.  Locally this satisfies 
$$\kappa(z,w)=\sum_{i,j=1}^mg^{ij}\cdot\frac{\partial z}{\partial x_i}\frac{\partial w}{\partial x_j}.$$
For the naming of the operator $\kappa$, we have the following.
\begin{remark}
Note that for a complex-valued function $z=x+i\,y:(M,g)\to\cn$ we have 
$$\kappa(z,z)=(|\nabla x|^2-|\nabla y|^2)+2i\, g(\nabla x,\nabla y)=0$$
if and only if the two gradients $\nabla x$ and $\nabla y$ are orthogonal and of the same length at every point $p\in M$ i.e. $z$ is {\it horizontally conformal}, see \cite{Bai-Woo-book}.
\end{remark}

As a direct consequence of the complex linearity and bi-linearity of the operators $\tau$ and $\kappa$, respectively, we have the following.

\begin{lemma}\label{lemma-conjugate}
Let $(M,g)$ be a Riemannian manifold and $z,w:(M,g)\to\cn$ be two complex-valued functions.  Then the tension field $\tau$ and the conformality operator $\kappa$ satisfy
\begin{equation}
\overline{\tau(z)}=\tau(\bar z)\ \ \text{and}\ \ \overline{\kappa(z,w)}=\kappa(\bar z,\bar w).
\end{equation}	 
\end{lemma}

We are now ready to define the complex-valued proper $p$-harmonic functions, the main objects of our study.

\begin{definition}\label{definition-proper-r-harmonic}
For a positive integer $p$, the iterated Laplace-Beltrami operator $\tau^p$ is given by
$$\tau^{0} (z)=z\ \ \text{and}\ \ \tau^p (z)=\tau(\tau^{(p-1)}(z)).$$	We say that a complex-valued function $z:(M,g)\to\cn$ is
\begin{enumerate}
\item[(a)] {\it $p$-harmonic} if $\tau^p (z)=0$, and
\item[(b)] {\it proper $p$-harmonic} if $\tau^p (z)=0$ and $\tau^{(p-1)} (z)$ does not vanish identically.
	\end{enumerate}
\end{definition}

We now introduce the natural notion of a $(p,q)$-harmonic morphism.  For $(p,q)=(1,1)$ this is the classical case of harmonic morphisms introduced by Fuglede and Ishihara, in \cite{Fug} and \cite{Ish}, independently.

\begin{definition}\label{definition-pq}
A map $\phi:(M,g)\to (N,h)$ between Riemannian manifolds is said to be a {\it $(p,q)$-harmonic morphism} if, for any {\it $q$-harmonic} function $f:U\subset N\to\rn$, defined on an open subset $U$ such that $\phi^{-1}(U)$ is not empty, the composition $f\circ\phi:\phi^{-1}(U)\subset M\to\rn$ is  {\it $p$-harmonic}.
\end{definition}

As an immediate consequence of Definition \ref{definition-pq} we have the following natural composition law.

\begin{lemma}\label{lemma-composition}
Let $\phi:(M,g)\to (\bar N,\bar h)$ be a $(p,r)$-harmonic morphism between Riemannian manifolds.  If $\psi:(\bar N,\bar h)\to (N,h)$ is an $(r,q)$-harmonic morphism then  the composition $\psi\circ\phi:(M,g)\to (N,h)$ is a $(p,q)$-harmonic morphism.
\end{lemma}

Another useful consequence of Definition \ref{definition-pq} is the following.

\begin{lemma}\label{lemma-higher-lower}
Let $\phi:(M,g)\to (N,h)$ be a $(p,q)$-harmonic morphism between Riemannian manifolds.  Then $\phi$ is a $(p,q^\prime)$-harmonic morphism for $q>q^\prime$ and is a $(p^\prime,q)$-harmonic morphism for $p^\prime >p$.
\end{lemma}

In our Theorem \ref{theorem-(p,q)}, we show that a complex-valued  $(p,q)$-harmonic morphism is a $p$-harmonic function and for this situation we define the following.

\begin{definition}A complex-valued $(p,q)$-harmonic morphism  $z:(M,g)\to\mathbb{C}$ from a Riemannian manifold is said to be {\it proper} if it is proper as a $p$-harmonic function i.e. $\tau^p(z)=0$ and $\tau^{p-1}(z)\neq 0$.
\end{definition}

\section{Complex-valued $(2,q)$-Harmonic Morphisms}

Throughout this work we assume that $z:(M,g)\to\cn$ is a smooth complex-valued function on a Riemannian manifold and that $f:U\to\cn$ is differentiable and defined on an open subset $U$ of $\cn$ containing the image $z(M)$ of $z$.  Further let $\phi:(M,g)\to\cn$ be the composition $\phi=f\circ z$.  For this situation we have the following result that later will be employed several times.  This is easily obtained by using the chain rule.

\begin{lemma}\label{lemma-fundamental}
Let $z:(M,g)\to\cn$ be a complex-valued function on a Riemannian manifold and $F,G:U\to\cn$ be differentiable functions defined on an open subset $U$ of $\cn$ containing the image $z(M)$ of $z$. Then the tension field $\tau$ and the conformality operator $\kappa$ satisfy
\begin{eqnarray*}
\tau(F(z,\bar z))&=&\frac{\partial F}{\partial z}\cdot\tau(z)+\frac{\partial F}{\partial \bar z}\cdot\tau(\bar z)\\
& &+\frac{\partial^2 F}{\partial z^2}\cdot\kappa(z,z)+
2\,\frac{\partial^2 F}{\partial z\partial\bar z}\cdot\kappa(z,\bar z)+\frac{\partial^2 F}{\partial \bar z^2}\cdot\kappa(\bar z,\bar z).
\end{eqnarray*}
and 
\begin{equation*}
\kappa(F(z,\bar z),G(z,\bar z))=\frac{\partial F}{\partial z}\cdot\kappa(z,G(z,\bar z))+\frac{\partial F}{\partial \bar z}\cdot\kappa (\bar z,G(z,\bar z)).
\end{equation*}
\end{lemma}

As a direct consequence of Lemma \ref{lemma-fundamental}, we now see that the tension field $\tau(\phi)$ of the composition $\phi=f\circ z$ is given by
\begin{equation}\label{equation-tau1}
\begin{aligned}
\tau(\phi)&=\frac{\partial f}{\partial z}\cdot\tau(z)+\frac{\partial f}{\partial \bar z}\cdot\tau(\bar z) \\
&\qquad +\frac{\partial^2 f}{\partial z^2}\cdot\kappa(z,z)+
2\,\frac{\partial^2 f}{\partial z\partial\bar z}\cdot\kappa(z,\bar z)+\frac{\partial^2 f}{\partial \bar z^2}\cdot\kappa(\bar z,\bar z).
\end{aligned}
\end{equation}

For the completeness of our exposition we now state the following.  This  recovers the classical result of Fuglede and Ishihara in our special case of complex-valued functions. 

\begin{theorem}\label{theorem-(1,1)}
A complex-valued function $z:(M,g)\to\cn$ from a Riemannian manifold is a $(1,1)$-harmonic morphism if and only if $$\kappa(z,z)=0\ \ \text{and}\ \ \tau(z)=0.$$
\end{theorem}

\begin{proof}
The function $z:(M,g)\to\cn$ is a $(1,1)$-harmonic morphism if and only if, for any harmonic $f:U\to\cn$ defined on an open subset $U$ of $\cn$ containing the image $z(M)$ of $z$, the tension field $\tau(\phi)$ of the composition $\phi=f\circ z$ vanishes. Since the function $f$ is assumed to be harmonic we have 
$$\tau(f)=\frac{\partial^2 f}{\partial z\partial\bar z}=0.$$
It now follows from Lemma \ref{lemma-conjugate} and equation (\ref{equation-tau1}) that $\tau(\phi)=0$ is equivalent to $$\kappa(z,z)=0\ \ \text{and}\ \ \tau(z)=0.$$
\end{proof}

\begin{proposition}\label{proposition-constant}
Let $z:(M,g)\to\cn$ be a complex-valued $(1,q)$-harmonic morphism from a Riemannian manifold. If $1<q$ then the function $z$ is constant.
\end{proposition}

\begin{proof}
The condition $1<q$ implies from (3.1) that both $\kappa(z,z)=0$ and $\kappa(z,\bar z)=0$ or equivalently that the function $z$ is constant.
\end{proof}

The next result is our fundamental tool for analysing the case of $(2,q)$.

\begin{lemma}\label{lemma-tau2}
Let $z:(M,g)\to\mathbb{C}$ be a complex-valued function from a Riemannian manifold and $f:U\to\mathbb{C}$ be defined on an open subset $U$ of $\mathbb{C}$ containing the image $z(M).$ Then the $2$-tension field $\tau^2(\phi)$ of the composition $\phi=f\circ z$ satisfies 
\begin{eqnarray*}
	& &\tau^2(\phi)\\
	&=&\tau^2(z)\cdot\frac{\partial f}{\partial z}+\tau^2(\bar z)\cdot\frac{\partial f}{\partial \bar z}\\
	& &+\bigl[\tau(z)^2+2\cdot\kappa(z,\tau(z))+\tau(\kappa(z,z))\bigr]\cdot\frac{\partial^2 f}{\partial z^2}\\
	& &+2\cdot \bigl[\tau(z)\tau(\bar z)+\kappa(z,\tau(\bar z))+\kappa(\bar z,\tau(z))+\tau(\kappa(z,\bar z))\bigr]\cdot\frac{\partial^2 f}{\partial z\partial\bar z}\\
	& &+\bigl[\tau(\bar z)^2+2\cdot\kappa(\bar z,\tau(\bar z))+\tau(\kappa(\bar z,\bar z))\bigr]\cdot\frac{\partial^2 f}{\partial\bar z^2}\\
	& &+2\cdot
	\bigl[\kappa(z,z)\tau(z)+\kappa(z,\kappa(z,z))\bigr]\cdot\frac{\partial^3 f}{\partial z^3}\\
	& &+2\cdot
	\bigl[2\cdot\kappa(z,\bar z)\tau(z)+\kappa(z,z)\tau(\bar z)\\
	& &\quad\quad  +\kappa(\bar z,\kappa(z,z))+2\cdot\kappa(z,\kappa(z,\bar z))\bigr]\cdot\frac{\partial^3 f}{\partial z^2\partial\bar z}\\
	& &+2\cdot\bigl[2\cdot\kappa(z,\bar z)\tau(\bar z) +\kappa(\bar z,\bar z)\tau(z)\\
	& &\quad\quad+\kappa(z,\kappa(\bar z,\bar z))+2\cdot\kappa(\bar z,\kappa(z,\bar z))
	\bigr]\cdot\frac{\partial^3 f}{\partial z\partial\bar z^2}\\
	& &+2\cdot
	\bigl[\kappa(\bar z,\bar z)\tau(\bar z)+\kappa(\bar z,\kappa(\bar z,\bar z))\bigr]\cdot\frac{\partial^3 f}{\partial\bar z^3}\\
	& &+\kappa(z,z)^2\cdot\frac{\partial^4 f}{\partial z^4}+4\cdot
	\kappa(z,z)\kappa(z,\bar z)\cdot\frac{\partial^4 f}{\partial z^3\partial\bar z}\\
	& &+2\cdot\bigl[\kappa(z,z)\kappa(\bar z,\bar z)+2\cdot\kappa(z,\bar z)^2\bigr]\cdot\frac{\partial^4 f}{\partial z^2\partial\bar z^2}\\
	& &+4\cdot
	\kappa(\bar z,\bar z)\kappa(z,\bar z)\cdot\frac{\partial^4 f}{\partial z\partial\bar z^3}+\kappa(\bar z,\bar z)^2\cdot\frac{\partial^4 f}{\partial \bar z^4}.\\
\end{eqnarray*}

\end{lemma}

\begin{proof}
Utilising the two basic equations (\ref{equation-basic}) and (\ref{equation-tau1}) we see that the $2$-tension field $\tau^2(\phi)$ of the composition $\phi=f\circ z$ satisfies 
\begin{eqnarray*}
& &\tau^2(\phi)\\
&=&\tau(\frac{\partial f}{\partial z})\cdot\tau(z)+2\cdot\kappa(\frac{\partial f}{\partial z},\tau(z))+\frac{\partial f}{\partial z}\cdot\tau^2(z)\\
& &\quad +\tau(\frac{\partial f}{\partial \bar z})\cdot\tau(\bar z)+2\cdot\kappa(\frac{\partial f}{\partial \bar z},\tau(\bar z))+\frac{\partial f}{\partial \bar z}\cdot\tau^2(\bar z)\\
& &\quad +\tau(\frac{\partial^2 f}{\partial z^2})\cdot\kappa(z,z)
+2\cdot\kappa(\frac{\partial^2 f}{\partial z^2},\kappa(z,z))
+\frac{\partial^2 f}{\partial z^2}\cdot\tau(\kappa(z,z))\\
& &\quad +2\cdot\tau(\frac{\partial^2 f}{\partial z\partial\bar z})\cdot\kappa(z,\bar z)
+4\cdot\kappa(\frac{\partial^2 f}{\partial z\partial\bar z},\kappa(z,\bar z))
+2\cdot\frac{\partial^2 f}{\partial z\partial\bar z}\cdot\tau(\kappa(z,\bar z))\\
& &\quad +\tau(\frac{\partial^2 f}{\partial \bar z^2})\cdot\kappa(\bar z,\bar z)
+2\cdot\kappa(\frac{\partial^2 f}{\partial \bar z^2},\kappa(\bar z,\bar z))
+\frac{\partial^2 f}{\partial \bar z^2}\cdot\tau(\kappa(\bar z,\bar z)).
\end{eqnarray*}
By applying Lemma 3.1 and reordering the terms we then obtain the stated result.
\end{proof}

For later use, we now reformulate Lemma \ref{lemma-tau2} and thereby show that the $2$-tension field $\tau^2(\phi)$ of $\phi$ can be presented in terms of the different partial derivatives of $f$ with coefficients determined by the functions $z$,$\bar z$ and their various tension fields.

\begin{lemma}\label{lemma-tau2-expansion}
	Let $z:(M,g)\to\mathbb{C}$ be a complex-valued function from a Riemannian manifold and $f:U\to\mathbb{C}$ be defined on an open subset $U$ of $\mathbb{C}$ containing the image $z(M).$ Then the $2$-tension field $\tau^2(\phi)$ of the composition $\phi=f\circ z$ satisfies 
	\begin{eqnarray*}
		& &\tau^2(\phi)\\
		&=&\tau^2(z)\cdot\frac{\partial f}{\partial z}+\tau^2(\bar z)\cdot\frac{\partial f}{\partial \bar z}\\
		& &+\bigl[\tfrac 12\tau^2(z^2)-z\, \tau^2(z)\bigr]\cdot\frac{\partial^2 f}{\partial z^2}\\
		& &+\bigl[\tau^2(z\bar{z})-\bar{z}\tau^2(z)-z\tau^2(\bar{z})\bigr]\cdot\frac{\partial^2 f}{\partial z\partial\bar z}\\
		& &+\bigl[\tfrac{1}{2}\tau^2(\bar{z}^2)-\bar{z}\tau^2(\bar{z})\bigr]\cdot\frac{\partial^2 f}{\partial\bar z^2}\\
		& &+
		\bigl[\tfrac{1}{6}\tau^2(z^3)-\tfrac{1}{2}z\tau^2(z^2)+\tfrac{1}{2}z^2\tau^2(z)\bigr]\cdot\frac{\partial^3 f}{\partial z^3}\\
		& &+
		\bigl[\frac 12\tau^2(z^2\bar z)-\frac 12\bar z\tau^2(z^2)+z\bar z\tau^2(z)- z\tau^2(z\bar z)+\frac 12z^2\tau^2(\bar z)\bigr]\cdot\frac{\partial^3 f}{\partial z^2\partial\bar z}\\
		& &+\bigl[\frac 12\tau^2(z\bar z^2)-\frac 12 z\tau^2(\bar z^2)+z\bar z\tau^2(\bar z)- \bar z\tau^2(z\bar z)+\frac 12 \bar z^2\tau^2(z)
		\bigr]\cdot\frac{\partial^3 f}{\partial z\partial\bar z^2}\\
		& &+
		\bigl[\tfrac{1}{6}\tau^2(\bar z^3)-\tfrac{1}{2}\bar z\tau^2(\bar z^2)+\tfrac{1}{2}\bar z^2\tau^2(\bar z)\bigr]\cdot\frac{\partial^3 f}{\partial\bar z^3}\\
		& &+\bigl[\tfrac{1}{24}\tau^2(z^4)-\tfrac{1}{6}z\tau^2(z^3)+\tfrac{1}{4}z^2\tau^2(z^2)-\tfrac{1}{6}z^3\tau^2(z)\bigr]\cdot\frac{\partial^4 f}{\partial z^4}+\\
		& &+ \bigl[\tfrac{1}{6}\tau^2(z^3\bar{z})-\tfrac{1}{6}\bar{z}\tau^2(z^3)+\tfrac{1}{2}z\bar z\tau^2(z^2)-\tfrac{1}{2}z\tau^2(z^2\bar{z})+\tfrac{1}{2}z^2\tau^2(z\bar{z})\bigr.\\
		& &\bigl. \quad \quad -\tfrac{1}{6}z^3\tau^2(\bar{z})-\tfrac{1}{2}z^2\bar{z}\tau^2(z)\bigr]\cdot\frac{\partial^4 f}{\partial z^3\partial\bar z}\\
		& &+\bigl[\tfrac{1}{4}\tau^2(z^2\bar{z}^2)+\tfrac{1}{4}\bar z^2\tau^2(z^2)+\tfrac{1}{4}z^2\tau^2(\bar z^2)-\tfrac{1}{2}\bar{z}\tau^2(z^2\bar z)-\tfrac{1}{2}z\bar{z}^2\tau^2(z)+z\bar{z}\tau^2(z\bar{z})\bigr.\\
		& &\quad\quad \bigl.-\tfrac{1}{2}z^2\bar{z}\tau^2(\bar{z})-\tfrac{1}{2}z\tau^2(\bar{z}^2z)\bigr]\cdot\frac{\partial^4 f}{\partial z^2\partial\bar z^2}\\
		& &+ \bigl[\tfrac{1}{6}\tau^2(z\bar{z}^3)-\tfrac{1}{6}z\tau^2(\bar z^3)+\tfrac{1}{2} z\bar z\tau^2(\bar z^2)-\tfrac{1}{2}\bar z\tau^2(\bar z^2 z)+\tfrac{1}{2}\bar z^2\tau^2(z\bar{z})\bigr.\\
		& &\bigl. \quad \quad -\tfrac{1}{6}z^3\tau^2(z)-\tfrac{1}{2}\bar z^2 z\tau^2(\bar z)\bigr]\cdot\frac{\partial^4 f}{\partial z\partial\bar z^3}\\
		& & +\bigl[\tfrac{1}{24}\tau^2(\bar z^4)-\tfrac{1}{6}\bar z\tau^2(\bar z^3)+\tfrac{1}{4}\bar z^2\tau^2(\bar z^2)-\tfrac{1}{6}\bar z^3\tau^2(\bar z)\bigr]\cdot\frac{\partial^4 f}{\partial \bar z^4}.\\
	\end{eqnarray*}
	
\end{lemma}

\begin{proof}
The statement follows directly by inserting the following identities, and their conjugates, into the formula given in Lemma \ref{lemma-tau2}.  For this see Lemma \ref{lemma-conjugate}.
\begin{eqnarray*}\label{1}
\quad \tau(z)^2+2\cdot\kappa(z,\tau(z))+\tau(\kappa(z,z))
=\frac 12\tau^2(z^2)-z\, \tau^2(z),
\end{eqnarray*}
\begin{eqnarray*}\label{2}
& &\quad\tau(z)\tau(\bar z)+\kappa(z,\tau(\bar z))+\kappa(\bar z,\tau(z))+\tau(\kappa(z,\bar z))\\ 
& &=\frac 12\bigl(\tau^2(z\bar z)-\tau^2(z)\bar z-z\tau^2(\bar z)\bigr),
\end{eqnarray*}
\begin{eqnarray*}\label{3}
2\,\tau(z)\kappa(z,z)+2\kappa(z,\kappa(z,z))=\tfrac{1}{6}\tau^2(z^3)-\tfrac{1}{2}z\tau^2(z^2)+\tfrac{1}{2}z^2\tau^2(z),
\end{eqnarray*}
\begin{eqnarray*}\label{4}
& &2\,\kappa(z,\bar z)\tau(z)+\kappa(\bar z,\kappa(z,z))+\tau(\bar z)\kappa(z,z)+2\,\kappa(z,\kappa(z,\bar z))	\\
&=&\frac 14\tau^2(z^2\bar z)-\frac 14\bar z\tau^2(z^2)+\frac 12z\bar z\tau^2(z)-\frac 12z\tau^2(z\bar z)+\frac 14z^2\tau^2(\bar z),\\
\end{eqnarray*}
\begin{eqnarray*}\label{5}
& &2\kappa(z,z)\kappa(\bar{z},\bar{z})+4\kappa(z,\bar{z})^2\\
&=&\tfrac{1}{4}\tau^2(z^2\bar{z}^2)+\tfrac{1}{4}\bar z^2\tau^2(z^2)+\tfrac{1}{4}z^2\tau^2(\bar z^2)
-\tfrac{1}{2}\bar{z}\tau^2(z^2\bar z)\\
& &-\tfrac{1}{2}z\bar{z}^2\tau^2(z)+z\bar{z}\tau^2(z\bar{z})
-\tfrac{1}{2}z^2\bar{z}\tau^2(\bar{z})
-\tfrac{1}{2}z\tau^2(\bar{z}^2z),
\end{eqnarray*}
\begin{eqnarray*}\label{6}
4\,\kappa(z,z)\kappa(z,\bar{z})
&=&\tfrac{1}{6}\tau^2(z^3\bar{z})-\tfrac{1}{6}\bar{z}\tau^2(z^3)+\tfrac{1}{2}z\bar z\tau^2(z^2)-\tfrac{1}{2}z\tau^2(z^2\bar{z})\\
& &+\tfrac{1}{2}z^2\tau^2(z\bar{z})-\tfrac{1}{6}z^3\tau^2(\bar{z})-\tfrac{1}{2}z^2\bar{z}\tau^2(z),
\end{eqnarray*}
\begin{eqnarray*}\label{7}
\kappa(z,z)^2=\tfrac{1}{24}\tau^2(z^4)-\tfrac{1}{6}z\tau^2(z^3)+\tfrac{1}{4}z^2\tau^2(z^2)-\tfrac{1}{6}z^3\tau^2(z).
\end{eqnarray*}

\end{proof}

In their paper \cite{Gha-Ou-1}, the authors introduce the notion of {\it generalised harmonic morphisms} between Riemannian manifolds.  These are exactly the $(2,1)$-harmonic morphisms in the sense of our Definition \ref{definition-pq}.  They give a characterisation of these objects between Riemannian manifolds.  In general this is rather complicated, see  Theorem 2.2 of \cite{Gha-Ou-1}.  In our context, of complex-valued functions, it is the following.

\begin{theorem}\label{theorem-(2,1)}
A complex-valued function $z:(M,g)\to\cn$ from a Riemannian manifold is a $(2,1)$-harmonic morphism if and only if 
$$\kappa(z,z)=0,\ \ \tau^2(z)=0\ \ \text{and}\ \ \tau^2(z^2)=0.$$
\end{theorem}

\begin{proof}
The function $z:(M,g)\to\cn$ is a $(2,1)$-harmonic morphism if and only if, for any harmonic $f:U\to\cn$ defined on an open subset $U$ of $\cn$ containing the image $z(M)$ of $z$, the $2$-tension field $\tau^2(\phi)$ of the composition $\phi=f\circ z$ vanishes.  It follows immediately from Lemma \ref{lemma-fundamental} that 
$$\kappa(z,z)=\kappa(\bar z,\bar z)=0.$$
Since the test function $f$ is assumed to be harmonic we also have 
$$\tau(f)=\frac{\partial^2 f}{\partial z\partial\bar z}=0.$$
This means that the formulae for the $2$-tension field $\tau^2(\phi)$, presented in Lemmas \ref{lemma-fundamental} and \ref{lemma-tau2-expansion}, simplify considerably.  The statement is then a direct consequence of the latter.
\end{proof}

\begin{remark}
	In the case when the Riemannian manifold $(M,g)$ is a surface, i.e. of dimension $2$, then the horizontal conformality of $\phi:M\to\cn$ and the Cauchy-Riemann equations imply harmonicity.  That means that in this case no proper $(2,1)$-harmonic morphisms do exist.
\end{remark}

In their paper \cite{Gha-Ou-1} the authors construct the following first known proper $(2,1)$-harmonic morphism.  This was basically the only known example before this current study.

\begin{example}\label{example-basic}
	Let $\rn^4$ be the standard $4$-dimensional Euclidean space and $U$ be the open subset given by 	$$U=\{(x_1,x_2,x_3,x_4)\in\rn^4|\ x_1^2+x_2^3+x_3^2>0\}.$$  Then $z:U\to\cn$ satisfying $z(x)= \sqrt{x_1^2+x_2^2+x_3^2}+i\, x_4$ is a {proper} $(2,1)$-harmonic morphism.
\end{example}

Furthermore they introduce several interesting general methods for constructing solutions to our non-linear $(2,1)$-problem from Euclidean spaces. The following result is a direct consequence of Corollary 3.1. of \cite{Gha-Ou-1}.

\begin{proposition}\label{proposition-holomorphic-composition}
Let $(M,g)$ be a Riemannian manifold and $z:M\to\cn$ be a $(2,1)$-harmonic morphism. Further let $f:U\to\cn$ be a holomorphic function defined on an open subset of $\cn$ such that $z(M)\subset U$.  Then the composition $f\circ z:M\to\cn$ is a $(2,1)$-harmonic morphism.
\end{proposition}

\begin{proof}
	It is a classical result that any such holomorphic function $f$ is a $(1,1)$-harmonic morphism. The statement then is a direct concequence of our Lemma \ref{lemma-composition}.
\end{proof}

The following is a special case of Lemma \ref{lemma-composition}, noted already in \cite{Gha-Ou-1}.

\begin{proposition}\label{proposition-known-composition-2}
	Let $(M,g)$ and $(N,h)$ be Riemannian manifolds, $f:(M,g)\to (N,h)$ be a $(2,2)$-harmonic morphism and $\phi:N\to\cn$ be a $(2,1)$-harmonic morphism.  Then the composition $\phi\circ f:(M,g)\to\cn$ is a $(2,1)$-harmonic morphism.
\end{proposition}

\begin{remark}\label{remark-proper}
	The reader should note that the word {\it "proper"} does not appear in Proposition \ref{proposition-known-composition-2}.  As we will see later, there is a good reason for this.
\end{remark}

From the above calculations of the $2$-tension field $\tau^2(\phi)$ we now have the following result in the case when $(p,q)=(2,2)$.  This should be compared with Theorem 4.1 of \cite{Lou-Ou-1} and Theorem 3.3 of \cite{Lou-Ou-2}.

\begin{theorem}\label{theorem-(2,2)}
A complex-valued function $z:(M,g)\to\cn$ from a Riemannian manifold is a $(2,2)$-harmonic morphism if and only if
$$\kappa(z,z)=0,\ \ \tau^2(z)=0,\ \ \tau^2(z^2)=0,$$ $$\tau^2(z\bar z)=0,\ \ \tau^2(z^2\bar z)=0.$$
\end{theorem}

\begin{proof}
The function $z:(M,g)\to\cn$ is a $(2,2)$-harmonic morphism if and only if, for any $2$-harmonic $f:U\to\cn$ defined on an open subset $U$ of $\cn$ containing the image $z(M)$ of $z$, the $2$-tension field $\tau^2(\phi)$ of the composition $\phi=f\circ z$ vanishes. It follows directly from Lemma \ref{lemma-fundamental} that 
$$\kappa(z,z)=\kappa(\bar z,\bar z)=0.$$
Since the function $f$ is assumed to be $2$-harmonic we also have 
$$\tau^2(f)=\frac{\partial^4 f}{\partial z^2\partial\bar z^2}=0.$$
This means that the formulae for the $2$-tension field $\tau^2(\phi)$, presented in Lemmas \ref{lemma-fundamental} and \ref{lemma-tau2-expansion}, simplify considerably.  The statement is then an immediate consequence of the latter.
\end{proof}

The next statement follows immediately from Proposition 3.2. of \cite{Gha-Ou-1}.

\begin{proposition}\label{proposition-sum}
	Let $(M,g)$, $(N,h)$ be Riemannian manifolds, $\phi:M\to\cn$ be a $(2,1)$-harmonic morphism and $\psi:N\to\cn$ be a $(1,1)$-harmonic morphism. Then the sum $\Phi=\phi\oplus\psi:M\times N\to\cn$, with  $$\Phi:(x,y)\mapsto \phi(x)+\psi(y),$$ is a $(2,1)$-harmonic morphism on the Riemannian product $M\times N$.
\end{proposition}

\section{New $(2,1)$-harmonic morphisms}

In this section we present several new proper complex-valued $(2,1)$-harm\-onic morphisms locally defined on Euclidean $\rn^n$. Example \ref{example-first} shows that such objects can easily be constructed for any dimension $n\ge 4$.

\begin{definition}
For a positive integer $p\in\zn^+$ we denote by $i_p$ the {\it inversion} $i_p:\rn^{2p}\setminus\{0\}\to\rn^{2p}\setminus\{0\}$ of the unit sphere $S^{2p-1}$ in $\rn^{2p}$ satisfying $$i_p(x)=\frac x{|x|^2}.$$
Let $\phi:U\to\cn$ be a function defined locally on an open subset $U$ of $\rn^{2p}\setminus\{0\}$. Then by its {\it dual function} $\phi^*$ we shall mean the composition $\phi^*=\phi\circ i_p:U\to\cn$.
\end{definition}

\begin{example}\label{example-first}
Let $\rn^n$ be the standard $n$-dimensional Euclidean space of dimension $n\ge 4$ and $U$ be the open subset given by 	$$U=\{x\in\rn^n|\ x_1^2+x_2^3+x_3^2>0\}.$$
Then the complex-valued function $\phi:U\to\cn$ defined by  
$$\phi(x)= \sqrt{x_1^2+x_2^2+x_3^2}+\sum_{k=4}^n b_k\cdot x_k$$ is a proper $(2,1)$-harmonic morphism if and only if the complex coefficients satisfy the relation
$$1+b_4^2+\cdots +b_n^2=0.$$
The same applies to the dual function $\phi^*=\phi\circ i_{p}$ in the case when $n=2p$.  It should be noted that, when $n=4$, the function $\phi$ reduces to Example 2 in \cite{Gha-Ou-1}.
\end{example}

\begin{example}\label{example-second}
Let $U$ be the open subset of the standard Euclidean space $\rn^4$ with $U=\{(x_1,x_2,x_3,x_4)\in\rn^4|\ x_2^2+x_3^2> 0\}$ and define the function $\phi:U\to\cn$ by			
$$\phi(x)=\frac{ x_2(1-|x|^2)+2\,x_1x_3}{x_2^2+x_3^2}+i\cdot\frac{x_3(1-|x|^2)-2\,x_1x_2}{x_2^2+x_3^2}.$$
Then $\phi$ is a proper $(2,1)$-harmonic morphism.  Furthermore, its dual function $\phi^*=\phi\circ i_2$ is the proper $(2,1)$-harmonic morphism given by  
$$\phi^*(x)=4\,x_1\cdot\frac{x_3-i\,x_2}{x_2^2+x_3^2}-\phi(x).$$ 
\end{example}	

\begin{example}\label{example-third}
The complex-valued function $\phi:\rn^4\setminus\{0\}\to\cn$ satisfying 		
$$\phi(x)=\log{\sqrt{x_1^2+x_2^2+x_3^2+x_4^2}}+i\cdot\arccos\Bigl({\frac{x_1}{\sqrt{x_1^2+x_2^2+x_3^2+x_4^2}}}\Bigr)$$ is a proper $(2,1)$-harmonic morphism. Its dual function $\phi^*=\phi\circ i_2$ is the proper $(2,1)$-harmonic morphism satisfying
$$\phi^*(x)=-\log{\sqrt{x_1^2+x_2^2+x_3^2+x_4^2}}+i\cdot\arccos\Bigl({\frac{x_1}{\sqrt{x_1^2+x_2^2+x_3^2+x_4^2}}}\Bigr).$$	
\end{example}

\begin{example}\label{example-fourth}
For a positive $r\in\rn^+$, the well-known local $(1,1)$-harmonic morphism $\phi_r:U\subset\rn^3\to\cn$, often called the {\it outer-disc example}, is given by
$$\phi_r(x)=\frac{-(x_3+ir)+\sqrt{x_1^2+x_2^2+x_3^2-r^2+2ir\cdot x_3}}{x_1-ix_2}.$$
Then the dual map $\phi_r^*=\phi_r\circ i_2$ satisfies 
$$\phi_r^*(x)=\frac{\sqrt{x_1^2+x_2^2+x_3^2+2ir\cdot x_3\cdot|x|^2-r^2\cdot|x|^4}-(x_3+ir\cdot|x|^2)}{x_1-ix_2}.$$
This is a proper $(2,1)$-harmonic morphism on $\rn^4$.
\end{example}

In the above Examples \ref{example-first}-\ref{example-fourth} we have seen that the  constructed complex-valued $(2,1)$-harmonic morphisms $\phi$ and its dual $\phi^*$ are both proper.  The next three examples show that this is not true in general, see Remark \ref{remark-proper}.

\begin{example}
For complex numbers $a,b,c,d\in\cn$ with $a^2+b^2+c^2+d^2=0$, let $\phi:\rn^4\setminus\{0\}\to\cn$ be the proper $(2,1)$-harmonic morphism 
$$\phi(x)= \frac{a\cdot x_1+b\cdot x_2+c\cdot x_3+d\cdot x_4}{x_1^2+x_2^2+x_3^2+x_4^2}.$$
Then its dual function 
$\phi^*=\phi\circ i_2$ is the globally defined $(1,1)$-harmonic morphism satisfying 
$$\phi^*:(x_1,x_2,x_3,x_4)\mapsto a\cdot x_1+b\cdot x_2+c\cdot x_3+d\cdot x_4.$$ 
By Lemma 2.5, this is a $(2,1)$-harmonic morphism, but it is not proper.
\end{example}

\begin{example}
For elements $a,b,c,d\in\cn$, define the complex-valued function $\phi:U\subset\rn^4\to\cn$ by
\begin{eqnarray*}
\phi(x)&=&\frac{a\cdot(x_1^2+x_2^2+x_3^2+x_4^2)+b\cdot(x_3+ix_4)}{x_1+ix_2}\\
& &\quad +\frac{c\cdot(x_1^2+x_2^2+x_3^2+x_4^2)+d\cdot(x_1+ix_2)}{x_3+ix_4}.
\end{eqnarray*}
Then $\phi$ is a proper $(2,1)$-harmonic morphism and its dual  $\phi^*=\phi\circ i_2$ is the holomorphic function
$$\phi^*(x)=\frac{
d\cdot(x_1+i\,x_2)^2+c\cdot(x_1+i\,x_2)+
b\cdot(x_3+i\,x_4)^2+a\cdot(x_3+i\,x_4)}
{(x_1+i\,x_2)\cdot(x_3+i\,x_4)}.
$$  This is clearly a $(2,1)$-harmonic morphism which is not proper.
\end{example}

\begin{example}
Define the complex-valued function $\phi:\rn^4\setminus\{0\}\to\cn$ by 
$$\phi(x)=\cos(\frac{x_1+i\,x_2}{x_1^2+x_2^2+x_3^2+x_4^2})+i\cdot\sin(\frac{x_3+i\,x_4}{x_1^2+x_2^2+x_3^2+x_4^2}).$$
Then $\phi$ is a proper $(2,1)$-harmonic morphism and its dual satisfying 
$$\phi^*(x)=\phi\circ i_2(x)=\cos(x_1+i\,x_2)+i\cdot\sin(x_3+i\,x_4)$$ 
is holomorphic and hence a $(2,1)$-harmonic morphism, but not proper as it is a $(1,1)$-harmonic morphism.
\end{example}

\section{A Generalised Construction Method}

The main purpose of this section is to prove Theorem \ref{theorem-general-2} which is a wide generalisation of Proposition \ref{proposition-sum}.

\begin{lemma}\label{lemma-hor-conf}
Let $(M,g)$, $(N,h)$ be Riemannian manifolds and $\phi:M\to\cn$, $\psi:N\to\cn$ be two horizontally conformal functions.  Let $U$ be an open subset of $\cn^2$ such that $\phi(M)\times \psi(N)\subset U$ 
and $f:U\to\cn$ be a holomorphic function.  Then the composition 
$\Phi:M\times N\to\cn$ with $\Phi(x,y)=f(\phi(x),\psi(y))$ is horizontally conformal on the Riemannian product space $M\times N$.
\end{lemma}

\begin{proof}
Let $\B_M$ and $\B_N$ be local orthonormal frames for the tangent bundles 
$TM$ and $TN$, respectively.  Then
\begin{eqnarray*}
\kappa(\Phi,\Phi)&=&\sum_{X\in\B_M}\bigl(\frac{\partial f}{\partial \phi}\cdot X(\phi)\bigr)^2+
\sum_{Y\in\B_N}\bigl(\frac{\partial f}{\partial \psi}\cdot Y(\psi)\bigr)^2\\
&=&\bigl(\frac{\partial f}{\partial \phi}\bigr)^2\cdot\kappa(\phi,\phi)+
\bigl(\frac{\partial f}{\partial \psi}\bigr)^2\cdot\kappa(\psi,\psi)\\
&=&0.
\end{eqnarray*}
\end{proof}

\begin{theorem}\label{theorem-general-2}
Let $U$, $V$ be open subsets of $\rn^m$ and $\rn^n$, respectively. Let $\phi:U\to\cn$ be a $(2,1)$-harmonic morphism and $\psi:V\to\cn$ be a $(1,1)$-harmonic morphism. Let $W$ be an open subset of $\cn^2$ such that $\phi(U)\times \psi(V)\subset W$ and 
$f:W\to\cn$ be a holomorphic function.  Then the composition $\Phi:U\times V\to\cn$ with $\Phi(x,y)=f(\phi(x),\psi(y))$ is a $(2,1)$-harmonic morphism.
\end{theorem}

\begin{proof}
It follows from Lemma \ref{lemma-hor-conf} that $\Phi$ is horizontally conformal i.e. $\kappa(\Phi,\Phi)=0$.  For the tension field $\tau(\Phi)$ of $\Phi$ we have
\begin{eqnarray*}
\tau(\Phi )&=&\sum_{k=1}^m\frac{\partial^2}{\partial x_k^2}(f(\phi(x),\psi(y)))+\sum_{r=1}^n\frac{\partial^2}{\partial y_r^2}(f(\phi(x),\psi(y)))\\	&=&\sum_{k=1}^m\frac{\partial}{\partial x_k}\bigl(\frac{\partial f}{\partial \phi}\cdot
\frac{\partial \phi}{\partial x_k}\bigr)+
\sum_{r=1}^n\frac{\partial}{\partial y_r}\bigl(\frac{\partial f}{\partial \psi}\cdot
\frac{\partial \psi}{\partial y_r}\bigr)\\
&=&\sum_{k=1}^m\bigl(\frac{\partial^2 f}{\partial \phi^2}\cdot\bigl(\frac{\partial \phi}{\partial x_k}\bigr)^2+\frac{\partial f}{\partial \phi}\cdot\frac{\partial^2\phi}{\partial x_k^2}\bigr)+
\sum_{r=1}^n\bigl(\frac{\partial^2 f}{\partial \psi^2}\cdot\bigl(\frac{\partial \psi}{\partial y_r}\bigr)^2+\frac{\partial f}{\partial \psi}\cdot\frac{\partial^2\psi}{\partial y_r^2}\bigr)\\
&=&\frac{\partial^2 f}{\partial \phi^2}\cdot\kappa(\phi,\phi)+\frac{\partial f}{\partial \phi}\cdot\tau(\phi)+\frac{\partial^2 f}{\partial \psi^2}\cdot\kappa(\psi,\psi)+\frac{\partial f}{\partial \psi}\cdot\tau(\psi)\\
&=&\frac{\partial f}{\partial \phi}\cdot\tau(\phi).
\end{eqnarray*}
With this at hand, we can now calculate  the $2$-tension field $\tau^2(\Phi)$ of $\Phi$ as follows.
\begin{eqnarray*}	
\tau^2(\Phi)&=&\tau(\frac{\partial f}{\partial \phi}\cdot\tau(\phi))\\
&=&\tau(\frac{\partial f}{\partial \phi})\cdot\tau(\phi)+2\cdot \kappa(\frac{\partial f}{\partial \phi},\tau(\phi))+\frac{\partial f}{\partial \phi}\cdot\tau^2(\phi)\\
&=&\tau(\phi)\cdot\Bigl(\sum_{k=1}^m\frac{\partial^2}{\partial x_k^2}\bigl(\frac{\partial f}{\partial \phi}\bigr)+\sum_{r=1}^n\frac{\partial^2}{\partial y_r^2}\bigl(\frac{\partial f}{\partial \phi}\bigr)\Bigr)\\
& &\quad+2\cdot\Bigl(\sum_{k=1}^m\frac{\partial}{\partial x_k}\bigl(\frac{\partial f}{\partial \phi}\bigr)\cdot\frac{\partial}{\partial x_k}(\tau(\phi))+\sum_{r=1}^n\frac{\partial}{\partial y_r}\bigl(\frac{\partial f}{\partial \phi}\bigr)\cdot\frac{\partial}{\partial y_r}(\tau(\phi))\Bigr)\\
&=&\tau(\phi)\cdot\Bigl(
\sum_{k=1}^m\frac{\partial}{\partial x_k}\bigl(\frac{\partial^2f}{\partial \phi^2}\cdot\frac{\partial \phi}{\partial x_k}\bigr)+\sum_{r=1}^n\frac{\partial}{\partial y_r}\bigl(\frac{\partial^2f}{\partial \psi\partial \phi}\cdot\frac{\partial \psi}{\partial y_r}\bigr)
\Bigr)\\
& &\quad +2\cdot\sum_{k=1}^m\frac{\partial^2f}{\partial \phi^2}\cdot\frac{\partial \phi}{\partial x_k}\cdot\frac{\partial}{\partial x_k}(\tau(\phi))\\
&=&\tau(\phi)\cdot
\sum_{k=1}^m\Bigl(\frac{\partial^3f}{\partial \phi^3}\cdot\bigl(\frac{\partial \phi}{\partial x_k}\bigr)^2+\frac{\partial^2f}{\partial \phi^2}\cdot\frac{\partial^2\phi}{\partial x_k^2}\Bigr)\\
& &\quad +\tau(\phi)\cdot\sum_{r=1}^n\Bigl(
\frac{\partial ^3f}{\partial \psi^2\partial \phi}\cdot\bigl(\frac{\partial \psi}{\partial y_r}\bigr)^2+\frac{\partial^2f}{\partial \psi\partial \phi}\cdot\frac{\partial^2 \psi}{\partial y_k^2}
\Bigr)\\
& &\quad\quad +2\cdot\frac{\partial^2 f}{\partial \phi^2}\cdot\kappa(\phi,\tau(\phi))\\
&=&
\tau(\phi)\cdot\Bigl(\frac{\partial^3 f}{\partial \phi^3}\cdot\kappa(\phi,\phi)+\frac{\partial ^3f}{\partial \psi^2\partial \phi}\cdot\kappa(\psi,\psi)+\frac{\partial^2f}{\partial \psi\partial \phi}\cdot\tau(\psi)\Bigr)\\
& &\quad +\frac{\partial ^2f}{\partial \phi^2}\cdot(2\cdot\kappa(\phi,\tau(\phi))+\tau(\phi)^2)\\
&=&0.
\end{eqnarray*}	
For the tension field $\tau(\Phi^2)$ of $\Phi^2$ we have 
$$\tau(\Phi^2)=2\cdot\Phi\cdot\tau(\Phi)+2\cdot\kappa(\Phi,\Phi)=2\cdot\Phi\cdot\tau(\Phi).$$
Hence the bi-tension field $\tau^2(\Phi^2)$ of $\Phi^2$ satisfies
\begin{eqnarray*}
\tau^2(\Phi^2)
&=&2\cdot\tau(\Phi\cdot\tau(\Phi))\\
&=&2\cdot(\tau(\Phi)^2+2\cdot\kappa(\Phi,\tau(\Phi))+\Phi\cdot\tau^2(\Phi))\\
&=&2\cdot\Bigl(\frac{\partial f}{\partial \phi}\Bigr)^2\cdot\tau(\phi)^2+4\cdot\kappa(\Phi,\frac{\partial \Phi}{\partial \phi}\cdot\tau(\phi))\\
&=&2\cdot\Bigl(\frac{\partial f}{\partial \phi}\Bigr)^2\cdot\tau(\phi)^2
+4\cdot\sum_{k=1}^m\frac{\partial\Phi}{\partial x_k}\cdot\frac{\partial}{\partial x_k}\bigl(\frac{\partial f}{\partial \phi}\cdot\tau(\phi)\Bigr)\\
& &\qquad\qquad +4\cdot\sum_{r=1}^n\frac{\partial\Phi}{\partial y_r}\cdot\frac{\partial}{\partial y_r}\bigl(\frac{\partial f}{\partial \phi}\cdot\tau(\phi)\Bigr)\\
&=&2\cdot\Bigl(\frac{\partial f}{\partial \phi}\Bigr)^2\cdot\tau(\phi)^2\\	
& &\quad +\,4\cdot\sum_{k=1}^m\frac{\partial \phi}{\partial x_k}\cdot\frac{\partial f}{\partial \phi}\cdot\Bigl(\frac{\partial^2 f}{\partial \phi^2}\cdot\frac{\partial \phi}{\partial x_k}\cdot\tau(\phi)+\frac{\partial f}{\partial \phi}\cdot\frac{\partial}{\partial x_k}(\tau(\phi))\Bigr)\\
& &\quad\quad +\,4\cdot\sum_{r=1}^n\frac{\partial \psi}{\partial y_r}\cdot\frac{\partial f}{\partial \psi}\cdot\Bigl(\frac{\partial^2 f}{\partial \psi\partial \phi}\cdot\frac{\partial \psi}{\partial y_r}\cdot\tau(\phi)+\frac{\partial f}{\partial \phi}\cdot\frac{\partial}{\partial y_r}(\tau(\phi))\Bigr)\\
&=&2\cdot\Bigl(\frac{\partial f}{\partial \phi}\Bigr)^2\cdot\tau(\phi)^2
+4\cdot\frac{\partial f}{\partial \phi}\cdot\frac{\partial^2 f}{\partial \phi^2}\cdot\tau(\phi)\cdot\kappa(\phi,\phi)\\
& &\quad +4\cdot\Bigl(\frac{\partial f}{\partial \phi}\Bigr)^2\cdot\kappa(\phi,\tau(\phi))+4\cdot\frac{\partial f}{\partial \psi}\cdot\frac{\partial^2 f}{\partial \psi\partial \phi}\cdot\tau(\phi)\cdot\kappa(\psi,\psi)\\
&=&2\cdot\Bigl(\frac{\partial f}{\partial \phi}\Bigr)^2 (\tau(\phi)^2+2\cdot\kappa(\phi,\tau(\phi)))\\
&=&0.
\end{eqnarray*}
The conclusion follows from Theorem 3.6.
\end{proof}

\begin{example}
We have already seen that the complex-valued function $\phi:\rn^4\setminus\{0\}\to\cn$ satisfying 		
$$\phi(x)=\log{\sqrt{x_1^2+x_2^2+x_3^2+x_4^2}}+i\cdot\arccos\Bigl({\frac{x_1}{\sqrt{x_1^2+x_2^2+x_3^2+x_4^2}}}\Bigr)$$ is a proper $(2,1)$-harmonic morphism.  It is clear that the holomorphic function $\psi:\rn^4\to\cn$ satisfying 
$$\psi(x)=\log(x_5+ix_6)\cdot\sin(x_7+ix_8)$$ is a $(1,1)$ harmonic morphism.
Calculations confirm that $\Phi:U\subset\rn^8\to\cn$ given by $\Phi=f(\phi,\psi)=\phi\cdot\psi$ is a proper $(2,1)$-harmonic morphism for any holomorphic function f.
\end{example}

\section{Complex-valued $(3,q)$-Harmonic Morphisms}

In this section we present a formula for the $3$-tension field $\tau^3(\phi)$, of the composition $\phi=f\circ z$. It turns out that, just as in the case of $(2,q)$, horizontal conformality, i.e. $\kappa(z,z)=0$, is a necessary condition.  Elementary but rather tedious calculations provide the following useful result.

\begin{lemma}\label{lemma-tau3-horizontally-conformal}
	Let $z:(M,g)\to\mathbb{C}$ be a horizontally conformal complex-valued function from a Riemannian manifold and $f:U\to\mathbb{C}$ be defined on an open subset $U$ of $\mathbb{C}$ containing the image $z(M).$ Then the $3$-tension field $\tau^3(\phi)$ of the composition $\phi=f\circ z$ satisfies 
	\begin{eqnarray*}
		& &\tau^3(\phi)\\
		&=&\tau^3(z)\cdot\frac{\partial f}{\partial z}+\tau^3(\bar{z})\cdot\frac{\partial f}{\partial\bar{z}}\\
		& & +\left[\tfrac{1}{2}\tau^3(z^2)-z\tau^3(z)\right]\cdot\frac{\partial^2f}{\partial z^2}\\
		& &
		+\left[\tau^3(z\bar{z})-\bar{z}\tau^3(z)-z\tau^3(\bar{z})\right]\cdot\frac{\partial^2f}{\partial z\partial\bar{z}}\\
		& & +\left[\tfrac{1}{2}\tau^3(\bar{z}^2)-\bar{z}\tau^3(\bar{z})\right]\cdot\frac{\partial^2f}{\partial \bar{z}^2}\\
		& &
		+\left[\tfrac{1}{6}\tau^3(z^3)-\tfrac{1}{2}z\tau^3(z^2)+\tfrac{1}{2}z^2\tau^3(z)\right]\cdot\frac{\partial ^3f}{\partial z^3}\\
		& &
		+\left[\tfrac{1}{2}\tau^3(z^2\bar{z})-\tfrac{1}{2}\bar{z}\tau^3(z^2)-z\tau^3(z\bar{z})+z\bar{z}\tau^3(z)+\tfrac{1}{2}z^2\tau^3(\bar{z})\right]\cdot\frac{\partial ^3f}{\partial{z^2}\partial\bar{z}}\\
		& &
		+\left[\tfrac{1}{2}\tau^3(z\bar{z}^2)-\bar{z}\tau^3(z\bar{z})+\tfrac{1}{2}\bar{z}^2\tau^3(z)-\tfrac{1}{2}z\tau^3(\bar{z}^2)+z\bar{z}\tau^3(\bar{z})\right]\cdot\frac{\partial ^3f}{\partial{z}\partial\bar{z}^2}\\
		& &
		+\left[\tfrac{1}{6}\tau^3(\bar{z}^3)-\tfrac{1}{2}\bar{z}\tau^3(\bar{z}^2)+\tfrac{1}{2}\bar{z}^2\tau^3(\bar{z})\right]\cdot\frac{\partial ^3f}{\partial \bar{z}^3}\\
		& &
		+\left[\tfrac{1}{6}\tau^3(z^3\bar{z})-\tfrac{1}{6}\bar{z}\tau^3(z^3)-\tfrac{1}{2}z\tau^3(z^2\bar{z})+\tfrac{1}{2}z\bar{z}\tau^3(z^2)+\tfrac{1}{2}z^2\tau^3(z\bar{z})-\tfrac{1}{2}z^2\bar{z}\tau^3(z)\right.\\
		& & \quad\quad\left.-\tfrac{1}{6}z^3\tau^3(\bar{z})\right]\cdot\frac{\partial ^4f}{\partial z^3\partial \bar{z}}\\
		& &
		+\left[\tfrac{1}{4}\tau^3(z^2\bar{z}^2)-\tfrac{1}{2}\bar{z}\tau^3(z^2\bar{z})+\tfrac{1}{4}\bar{z}^2\tau^3(z^2)-\tfrac{1}{2}z\tau^3(\bar{z}^2z)+z\bar{z}\tau^3(z\bar{z})-\tfrac{1}{2}z\bar{z}^2\tau^3(z)\right.\\
		& &
		\quad\quad\left. +\tfrac{1}{4}z^2\tau^3(\bar{z}^2)-\tfrac{1}{2}z^2\bar{z}\tau^3(\bar{z})\right]\cdot\frac{\partial ^4f}{\partial z^2\partial\bar{z}^2}\\
		& &
		+\left[\tfrac{1}{6}\tau^3(z\bar{z}^3)-\tfrac{1}{2}\bar{z}\tau^3(\bar{z}^2z)+\tfrac{1}{2}\bar{z}^2\tau^3(z\bar{z})-\tfrac{1}{6}\bar{z}^3\tau^3(z)-\tfrac{1}{6}z\tau^3(\bar{z}^3)+\tfrac{1}{2}z\bar{z}\tau^3(\bar{z}^2)\right.\\
		& &
		\quad\quad\left.-\tfrac{1}{2}z\bar{z}^2\tau^3(\bar{z})\right]\cdot\frac{\partial ^4f}{\partial z\partial \bar{z}^3}\\
		& &
		+\left[\tfrac{1}{12}\tau^3(z^3\bar{z}^2)-\tfrac{1}{6}\bar{z}\tau^3(z^3\bar{z})+\tfrac{1}{12}\bar{z}^2\tau^3(z^3)-\tfrac{1}{4}z\tau^3(z^2\bar{z}^2)+\tfrac{1}{2}z\bar{z}\tau^3(z^2\bar{z})-\tfrac{1}{4}z\bar{z}^2\tau^3(z^2)\right.\\
		& &\left.\quad +\tfrac{1}{4}z^2\tau^3(z\bar{z}^2)-\tfrac{1}{2}z^2\bar{z}\tau^3(z\bar{z})+\tfrac{1}{4}z^2\bar{z}^2\tau^3(z)-\tfrac{1}{12}z^3\tau^3(\bar{z}^2)+\tfrac{1}{6}z^3\bar{z}\tau^3(\bar{z})\right]\cdot\frac{\partial ^5f}{\partial z^3\partial\bar{z}^2}\\
		& &
		+\left[\tfrac{1}{12}\tau^3(z^2\bar{z}^3)-\tfrac{1}{4}\bar{z}\tau^3(z^2\bar{z}^2)+\tfrac{1}{4}\bar{z}^2\tau^3(z^2\bar{z})-\tfrac{1}{12}\bar{z}^3\tau^3(z^2)-\tfrac{1}{6}z\tau^3(z\bar{z}^3)+\tfrac{1}{2}z\bar{z}\tau^3(z\bar{z}^2)\right.\\
		& &\left.\quad -\tfrac{1}{2}z\bar{z}^2\tau^3(z\bar{z})+\tfrac{1}{6}z\bar{z}^3\tau^3(z)+\tfrac{1}{12}z^2\tau^3(\bar{z}^3)-\tfrac{1}{4}z^2\bar{z}\tau^3(\bar{z}^2)+\tfrac{1}{4}z^2\bar{z}^2\tau^3(\bar{z})\right]\cdot\frac{\partial ^5f}{\partial z^2\partial\bar{z}^3}\\
		& &
		+8\,\kappa(z,\bar{z})^3\cdot\frac{\partial^6f}{\partial z^3\partial\bar{z}^3}\\
	\end{eqnarray*}
\end{lemma}

\begin{theorem}\label{theorem-(3,1)}
A complex-valued function $z:(M,g)\to\mathbb{C}$ from a Riemannian manifold is a $(3,1)$-harmonic morphism if and only if
$$
\kappa(z,z)=0,
$$
$$\tau^3(z)=0,\ \ \tau^3(z^2)=0,\ \ \tau^3(z^3)=0,$$ 
\end{theorem}

\begin{proof}
The method used here is exactly the same as that we have employed in the proof of Theorem \ref{theorem-(2,1)} employing the fact that $f$ is harmonic i.e. 
$$\tau(f)=\frac{\partial^2 f}{\partial z\partial \bar z}=0.$$
\end{proof}

\begin{example}
Let $\phi:\rn^6\setminus\{0\}\to\cn$ be the $(1,1)$-harmonic morphism given by
$$\phi(x)=(x_1+ix_2)(x_3+ix_4)+\sin (x_5+ix_6).$$  
Then calculations show that its dual map $\phi^*=\phi\circ i_3:\rn^6\setminus\{0\}\to\cn$ is a proper $(3,1)$-harmonic morphism.
\end{example}

In the following Tables 1 and 2 we give several new examples of $(3,1)$-harmonic morphisms defined on the appropriate open subsets $U$ of $\rn^6$.  They are proper if and only if the stated $p$ is $3$ and not proper otherwise.

\begin{table}[ht]
\begin{center}
\caption{$\phi:U\subset\rn^6\setminus\{0\}\to\cn$.}
\label{tab:table1}
\begin{tabular}{|c|c|}	
\hline
$\phi(x)$ & $(p,q)$ \\
\hline\hline
$\phi_{11}(x)=(x_1+ix_2)+(x_3+ix_4)+(x_5+ix_6)$ & $(1,1)$ \\
$\phi_{12}(x)=\sqrt{x_1^2+x_2^2+x_3^2}+ix_4+ (x_5+ix_6)$ & $(2,1)$ \\
$\phi_{13}(x)=\sqrt{x_1^2+x_2^2+x_3^2+x_4^2+x_5^2}+ix_6$ & $(3,1)$ \\ 
\hline
$\phi^*_{11}(x)=\phi_{11}\circ i_3(x)$ & $(3,1)$ \\
$\phi^*_{12}(x)=\phi_{12}\circ i_3(x)$ & $(3,1)$ \\
$\phi^*_{13}(x)=\phi_{13}\circ i_3(x)$ & $(3,1)$ \\
\hline
\end{tabular}
\end{center}
\end{table}	

\begin{table}[ht]
\begin{center}
\caption{$\phi:U\subset\rn^6\setminus\{0\}\to\cn$.}
\label{tab:table2}
\begin{tabular}{|c|c|}	
\hline
$\phi(x)$ & $(p,q)$ \\
\hline\hline
$\phi_{21}(x)=\log\sqrt{x_1^2+x_2^2}+i\arccos (\frac {x_1} {\sqrt{x_1^2+x_2^2}})$ 
& $(1,1)$ \\			
$\phi_{22}(x)=\log\sqrt{x_1^2+\cdots +x_4^2}+i\arccos (\frac {x_1} {\sqrt{x_1^2+\cdots +x_4^2}})$ 
& $(2,1)$ \\			
$\phi_{23}(x)=\log\sqrt{x_1^2+\cdots +x_6^2}+i\arccos (\frac {x_1} {\sqrt{x_1^2+\cdots +x_6^2}})$ 
& $(3,1)$ \\			
\hline
$\phi^*_{21}(x)=\phi_{21}\circ i_3(x)$ & $(3,1)$ \\
$\phi^*_{22}(x)=\phi_{22}\circ i_3(x)$ & $(3,1)$ \\
$\phi^*_{23}(x)=\phi_{23}\circ i_3(x)$ & $(3,1)$ \\
\hline
\end{tabular}
\end{center}
\end{table}

\begin{theorem}\label{theorem-(3,2)}
	A complex-valued function $z:(M,g)\to\mathbb{C}$ from a Riemannian manifold is a $(3,2)$-harmonic morphism if and only if
	$$
	\kappa(z,z)=0,
	$$
	$$\tau^3(z)=0,\ \ \tau^3(z^2)=0,\ \ \tau^3(z^3)=0,$$ 
	$$\tau^3(z\bar{z})=0,\ \ \tau^3(z^2\bar{z})=0,\ \ \tau^3(z^3\bar{z})=0.$$
\end{theorem}

\begin{proof}
The statement follows easily from the fact that $\kappa(z,z)=\kappa(\bar z,\bar z)=0$, Lemma \ref{lemma-tau3-horizontally-conformal} and 
$$
\tau^2(f)=\frac{\partial^4f}{\partial z^2\partial \bar{z}^2}=0,
$$
since $f$ is assumed to be a general 2-harmonic function.
\end{proof}

Our next result gives a characterisation in the complex-valued $(3,3)$-case.  This recovers a particular case of Corollary 6.2 of the interesting work \cite{Mae} of Maeta.

\begin{theorem}\label{theorem-(3,3)}
	A complex-valued function $z:(M,g)\to\mathbb{C}$ from a Riemannian manifold is a $(3,3)$-harmonic morphism if and only if
	$$
	\kappa(z,z)=0,
	$$
	$$\tau^3(z)=0,\ \ \tau^3(z^2)=0,\ \  \tau^3(z^3)=0,
	$$ 
	$$
	\tau^3(z\bar{z})=0,\ \ \tau^3(z^2\bar{z})=0,\ \  \tau^3(z^3\bar{z})=0,
	$$
	$$
	\tau^3(z^2\bar{z}^2)=0,\ \  \tau^3(z^3\bar{z}^2)=0.
	$$
\end{theorem}

\begin{proof}
Here we use exactly the same method as above, utilising Lemma \ref{lemma-tau3-horizontally-conformal} and the fact that in this case we have $$\tau^3(f)=\frac{\partial^6f}{\partial z^3\partial\bar z^3}=0.$$
\end{proof}

\section{Complex-valued $(p,q)$-harmonic morphisms}

In this section we investigate the $p$-tension field $\tau^p(\phi)$ of the composition $\phi=f\circ z$ and derive several consequences from  the condition $\tau^p(\phi)=0$ i.e. of $\phi$ being $p$-harmonic.  It turns out that $\tau^p(\phi)$ takes the following form
$$\tau^p(\phi)=\displaystyle{
	\sum\limits_{\substack{ 1\leq j+k\leq 2p}}
	c_{jk}^p\cdot\frac{\partial^{j+k}f}{\partial z^j\partial\bar{z}^{k}}},$$
where the coefficients $c^p_{jk}:U\to\cn$ are differentiable functions involving various tension fields and conformality operators of the functions  $z$ and $\bar z$ and independent of $f$.

We have already presented the tension fields $\tau(\phi)$, $\tau^2(\phi)$ and $\tau^3(\phi)$ of $\phi$. When calculating the $4$-tension field $\tau^4(\phi)$ a clear pattern comes to light.  These calculations are far too extensive to be presented here.  For the $p$-tension field $\tau^p(\phi)$ we have the following  result.

\begin{lemma}\label{lemma-pq}
Let $z:(M,g)\to\mathbb{C}$ be a complex-valued function from a Riemannian manifold and $f:U\to\mathbb{C}$ be defined on an open subset $U$ of $\mathbb{C}$ containing the image $z(M).$ Then for $p\geq 2$ the $p$-tension field $\tau^p(\phi)$ of the composition $\phi=f\circ z$ is of the form
$$\tau^p(\phi)=\displaystyle{
\sum\limits_{\substack{1\leq j+k\leq 2p}}
c_{jk}^p\cdot\frac{\partial^{j+k}f}{\partial z^j\partial\bar{z}^{k}}}.$$
The coefficients $c^p_{jk}:U\to\cn$ are differentiable functions, independent of $f$, involving various tension fields and conformality operators of the functions $z$ and $\bar z$. Furthermore, they are symmetric with respect to complex conjugation i.e. $c_{jk}^p=\bar c_{kj}^p$ and 
$$c_{10}=\tau^p(z),\ \ c_{01}=\tau^p(\bar z),\ \ c_{2p,0}=\kappa(z,z)^p,\ \  c_{0,2p}=\kappa(\bar z,\bar z)^p.$$
\end{lemma}

This leads to the following general result which should be compared with Theorems \ref{theorem-(1,1)}, \ref{theorem-(2,1)}, \ref{theorem-(2,2)}, \ref{theorem-(3,1)}, \ref{theorem-(3,2)} and \ref{theorem-(3,3)} above.

\begin{theorem}\label{theorem-(p,q)}
A complex-valued function $z:(M,g)\to\mathbb{C}$ from a Riemannian manifold is a $(p,q)$-harmonic morphism if and only if
$$\kappa(z,z)=0,$$
$$\tau^p(z)=0,\ \ \tau^p(z^2)=0,\ \ \cdots\ \  ,\tau^p(z^p)=0,$$
$$\tau^p(z\bar{z})=0,\ \ \tau^p(z^2\bar{z})=0,\ \  \cdots\ \ ,\tau^p(z^p\bar{z})=0,$$
$$\tau^p(z^2\bar{z}^2)=0,\ \ \tau^p(z^3\bar{z}^2)=0,\ \ \cdots\ \ ,\tau^p(z^p\bar{z}^2)=0,$$
$$\vdots$$
$$
\tau^p(z^{q-1}\bar{z}^{q-1})=0,\ \ \tau^p(z^{q}\bar{z}^{q-1})=0,\ \ \cdots\ \ , \tau^p(z^p\bar{z}^{q-1})=0.$$
\end{theorem}

\begin{proof}
The function $z:(M,g)\to\cn$ is a $(p,q)$-harmonic morphism if and only if, for any $q$-harmonic function $f:U\to\cn$ defined on an open subset $U$ of $\cn$ containing the image $z(M)$ of $z$, the $p$-tension field $\tau^p(\phi)$ of the composition $\phi=f\circ z$ vanishes. Since the function $f$ is assumed to be $q$-harmonic we know that 
$$\tau^q(f)=\frac{\partial^{2q}f}{\partial z^q\partial\bar z^q}=0.$$
According to Lemma \ref{lemma-pq} we also have 
$$\tau^p(z)=\tau^p(\bar z)=0\ \ \text{and}\ \ 
\kappa(z,z)=\kappa(\bar z,\bar z)=0.$$  If we now plug these indentities into the expression for $\tau^p(\phi)$ this simplifies considerably to
$$\tau^p(\phi)=\displaystyle{
	\sum\limits_{\substack{0\le j,k\le p\\ 2\leq j+k\leq 2p}}
	c_{jk}^p\cdot\frac{\partial^{j+k}f}{\partial z^j\partial\bar{z}^{k}}},$$
where 
$c_{pp}^p=2^p\cdot\kappa(z,\bar z)^p$. Hard work then shows that the remaining coefficients satisfy 
$$
c_{jk}^p=\sum\limits_{\substack{0\leq r\leq j\\ 0\leq s\leq k}}(-1)^{j-r+k-s}\frac{1}{j!}\frac{1}{k!}\binom jr\binom ks
z^{j-r}\bar{z}^{k-s}\tau^p(z^{r}\bar{z}^{s}).
$$
The rest follows by the same method as applied in the proof of Theorem \ref{theorem-(2,2)}.
\end{proof}

\begin{remark}
In the paper \cite{Mae}, Maeta presents his interesting Conjecture 7.6.  In our language his statement is: {\it "A $(p,p)$-harmonic morphism is characterized as a special horizontally weakly conformal $2p$-harmonic map."}  

In our Theorem \ref{theorem-(p,q)} we study the special case of complex-valued $(p,p)$-harmonic morphisms.  We obtain a characterisation of these  objects and show that they are both horizontally conformal and $2p$-harmonic, as Maeta suggests. But additionally, they must satisfy several rather non-trivial conditions.  They can therefore rightly be said to be "special horizontally weakly conformal $2p$-harmonic maps".
\end{remark}

We conclude this section by presenting further examples. They suggest that one should be able to produce $(p,1)$-harmonic morphisms for any positive integer $p\in\zn^+$.  The question marks '?' in Table \ref{tab:table4} tell us that the calculations needed, in those cases, were too heavy for the tools available to us.

\begin{table}[ht]
	\begin{center}
		\caption{$\phi:U\subset\rn^8\setminus\{0\}\to\cn$.}
		\label{tab:table3}
		\begin{tabular}{|c|c|}	
			\hline
			$\phi(x)$ & $(p,q)$ \\
			\hline\hline
			$\phi_{31}(x)=(x_1+ix_2)+(x_3+ix_4)+(x_5+ix_6)+(x_7+ix_8)$ & $(1,1)$ \\
			
			$\phi_{32}(x)=\sqrt{x_1^2+x_2^2+x_3^2}+ix_4+ (x_5+ix_6)+(x_7+ix_8)$ & $(2,1)$ \\
			
			$\phi_{33}(x)=\sqrt{x_1^2+x_2^2+x_3^2+x_4^2+x_5^2}+ix_6+(x_7+ix_8)$ & $(3,1)$ \\ 
			$\phi_{34}(x)=\sqrt{x_1^2+x_2^2+x_3^2+x_4^2+x_5^2+x_6^2+x_7^2}+ix_8)$ & $(4,1)$ \\ 
			\hline
			$\phi^*_{31}(x)=\phi_{31}\circ i_4(x)$ & $(4,1)$ \\
			$\phi^*_{32}(x)=\phi_{32}\circ i_4(x)$ & $(4,1)$ \\
			$\phi^*_{33}(x)=\phi_{33}\circ i_4(x)$ & $(4,1)$ \\
			$\phi^*_{34}(x)=\phi_{34}\circ i_4(x)$ & $(4,1)$ \\
			\hline
		\end{tabular}
	\end{center}
\end{table}

\begin{table}[ht]
	\begin{center}
\caption{$\phi:U\subset\rn^8\setminus\{0\}\to\cn$.}
		\label{tab:table4}
		\begin{tabular}{|c|c|}	
\hline
$\phi(x)$ & $(p,q)$ \\
\hline\hline
$\phi_{41}(x)=\log\sqrt{x_1^2+x_2^2}+i\arccos (\frac {x_1} {\sqrt{x_1^2+x_2^2}})$ & $(1,1)$ \\			
$\phi_{42}(x)=\log\sqrt{x_1^2+\cdots +x_4^2}+i\arccos (\frac {x_1} {\sqrt{x_1^2+\cdots +x_4^2}})$ & $(2,1)$ \\			
$\phi_{43}(x)=\log\sqrt{x_1^2+\cdots +x_6^2}+i\arccos (\frac {x_1} {\sqrt{x_1^2+\cdots +x_6^2}})$ & $(3,1)$ \\
$\phi_{44}(x)=\log\sqrt{x_1^2+\cdots +x_8^2}+i\arccos (\frac {x_1} {\sqrt{x_1^2+\cdots +x_8^2}})$ & $(4,1)$ \\
\hline					
$\phi^*_{41}(x)=\phi_{41}\circ i_4(x)$ & $(4,1)$ \\
$\phi^*_{42}(x)=\phi_{42}\circ i_4(x)$ &  ? \\
$\phi^*_{43}(x)=\phi_{43}\circ i_4(x)$ &  ? \\
$\phi^*_{44}(x)=\phi_{44}\circ i_4(x)$ &  ? \\
			\hline
		\end{tabular}
	\end{center}
\end{table}

\begin{remark}
	In the process of obtaining Lemma \ref{lemma-pq}, it is easily seen that every $(p,q)$-harmonic morphism is constant in the cases when $p<q$.  This is due to the fact that in these cases we have 
$$c_{2p,0}^p=\kappa(z,z)=0\ \ \text{and}\ \ c_{pp}^p=\kappa(z,\bar z)=0.$$
	The reader should compare this with Proposition  \ref{proposition-constant}.
\end{remark}

\begin{example}
	Let $\phi:\rn^8\to\cn$ be the holomorphic $(1,1)$-harmonic morphism defined by 
	$$\phi(x)=(x_1+ix_2+x_3+ix_4)+\sin(x_5+ix_6+x_7+ix_8).$$
	Then its dual map $\phi^*=\phi\circ i_4:\rn^8\setminus\{0\}\to\cn$ is a proper $(4,1)$-harmonic morphism.
\end{example}

\section{The Inversion about the unit sphere $S^{2p-1}$ in $\rn^{2p}$}

In this section we investigate the inversion $i_p:\rn^{2p}\setminus
\{0\}\to\rn^{2p}\setminus\{0\}$ about the unit sphere $S^{2p-1}$ in $\rn^{2p}$.

\begin{theorem}\label{theorem-inversion}
Let $i_p:\rn^{2p}\setminus\{0\}\to\rn^{2p}\setminus\{0\}$ be the inversion about the unit sphere $S^{2p-1}$ in $\rn^{2p}$ given by 
$$i_p=(F_1,\dots ,F_{2p}):x\mapsto\frac {(x_1,\dots ,x_{2p})}{|x|^2}.$$  Then the map $i_p$ is horizontally conformal and $p$-harmonic.
\end{theorem}

\begin{proof}
The fact that $i_p$ is conformal is classic, but we prove it here for the reader's convenience.  For $1\le j,k\le 2p$ the conformality operator $\kappa$ satisfies
	\begin{eqnarray*}
		\kappa (F_j,F_k)
		&=&\sum_{s=1}^{2p}\frac{\partial F_j}{\partial x_s}\cdot\frac{\partial F_k}{\partial x_s}\\
		&=&\sum_{s=1}^{2p}
		\frac{(\delta_{js}|x|^2-2\,x_jx_s)}{|x|^4}\cdot
		\frac{(\delta_{ks}|x|^2-2\,x_kx_s)}{|x|^4}\\
		&=&\sum_{s=1}^{2p}
		\frac{(\delta_{js}\delta_{ks}|x|^4-2\,\delta_{js}x_kx_s|x|^2-2\,\delta_{ks}x_jx_s|x|^2+4x_jx_kx_s^2)}{|x|^8}\\
		&=&\frac{\delta_{jk}|x|^4-2|x|^2x_jx_k-2|x|^2x_kx_j+4x_jx_k|x|^2}{|x|^8}\\
		&=&\frac{\delta_{jk}}{|x|^4}.
	\end{eqnarray*}
	
	The fact that the map $i_p$ is proper $p$-harmonic is a direct consequence of the following repeated application of Lemma \ref{lemma-Laplacian}.
	$$\tau( i_p)=\frac{2(2-2p)}{|x|^2}\cdot i_p,$$
	$$\tau^2( i_p)=\frac{2(2-2p)4(4-2p)}{|x|^4}\cdot i_p$$
	$$\vdots$$
	
	$$\tau^p( i_p)=\frac{2(2-2p)4(4-2p)\cdots 2p(2p-2p)}{|x|^{2p}}\cdot i_p=0.$$
\end{proof}

\begin{lemma}\label{lemma-Laplacian}
	For a positive integer $n\in\zn^+$ let the map  $\phi:\rn^{p}\setminus\{0\}\to\rn^{p}\setminus\{0\}$ be given by 
	$$\phi=(\phi_1,\dots ,\phi_p):x\mapsto\frac {(x_1,\dots ,x_p)}{|x|^n}.$$ 
	Then the tension field $\tau(\phi)$ of $\phi$ satisfies
	$$\tau(\phi)=\frac {n(n-p)}{|x|^{n+2}}\cdot\phi.$$
\end{lemma}

\begin{proof}
First we notice that 
$$\frac{\partial}{\partial x_j}|x|^n=n\, x_j|x|^{n-2}.$$
Applying this several times we then get
$$\frac{\partial\phi_k}{\partial x_j}=\frac{\delta_{jk}|x|^n-nx_kx_j\, |x|^{n-2}}{|x|^{2n}}$$
and 
\begin{eqnarray*}
\frac{\partial^2\phi_k}{\partial x_j^2}
&=&\frac 1{|x|^{2n+2}}\Bigl(
\frac{\delta_{jk}nx_j|x|^{n-2}-\delta_{jk}nx_j|x|^{n-2}}{|x|^{2n+2}}\\
& &
\qquad\qquad  -\frac{nx_k\,|x|^{n-2}+n(n-2)x_kx_j^2|x|^{n-4}}{|x|^{2n+2}}\\
& &\qquad\qquad\qquad -\frac{2n(\delta_{jk}|x|^n-nx_kx_j\, |x|^{n-2}|x|^{2n})x_j}{|x|^{2n+2}}\Bigr).
\end{eqnarray*}	
This means that for the tension field $\tau(\phi_k)$ we yield 
\begin{eqnarray*}
& &\tau(\phi_k)\\
&=&\frac{-npx_k|x|^{3n-2}-n(n-2)x_k|x|^{3n-2}-2nx_k|x|^{3n-2}+2n^2x_k|x|^{3n-2}}{|x|^{4n}}\\
&=&(2n^2-2n-n(n-2)-np)\frac {x_k}{|x|^{n+2}}\\
&=&n(n-p)\frac{x_k}{|x|^{n+2}}.
\end{eqnarray*}		
\end{proof}

\section{Two Conjectures}

We conclude this work with two conjectures that have come to our minds while working on this project.

\begin{conjecture}\label{conjecture-1}
Let $p\in\zn^+$ be a positive integer and $i_p=(F_1,F_2,\dots ,F_{2p}):\rn^{2p}\setminus\{0\}\to\rn^{2p}\setminus\{0\}$ be the inversion about the unit sphere $S^{2p-1}$ in $\rn^{2p}$.  Then $z:\rn^{2p}\setminus\{0\}\to\cn$ with 
$$z=a_1F_1+a_2F_2\cdots+a_{2p}F_{2p}$$ is a complex-valued $(p,p)$-harmonic morphism for any element $a\in\cn^{2p}$.
\end{conjecture}

Our rather extensive computer calculations show that this Conjecture \ref{conjecture-1} is true in the cases when $p=1,2,3,4$, but the statement seems to be difficult to prove in general.
\medskip

No proper $(2,1)$-harmonic morphism is known to exist from the three dimensional Euclidean spaces $\rn^{3}$, not even locally.  For this we have the following.

\begin{conjecture}
Let $p\ge 2$ and $\phi:U\to\cn$ be a complex-valued $(p,1)$-harmonic morphism defined locally on the standard Euclidean space $\rn^{2p-1}$.  Then $\phi$ is a $(1,1)$-harmonic morphism i.e. $\tau(\phi)=0$.
\end{conjecture}

\section{Acknowledgements}

The authors would like to thank the referee for several useful remarks that have improved the presentation. The first author would like to thank the Department of Mathematics at Lund University for its great hospitality during her time there as a postdoc.

\end{document}